\documentclass[11pt,reqno]{amsart}
\usepackage{amsmath,amssymb,amsthm,mathtools}
\usepackage[margin=1.15in]{geometry}
\usepackage{graphicx}
\usepackage[colorlinks=true,citecolor=blue,linkcolor=blue]{hyperref}

\theoremstyle{plain}
\newtheorem{theorem}{Theorem}[section]
\newtheorem{proposition}[theorem]{Proposition}
\newtheorem{lemma}[theorem]{Lemma}

\newtheorem{conjecture}[theorem]{Conjecture}
\newtheorem{hypothesis}[theorem]{Hypothesis}
\theoremstyle{definition}
\newtheorem{definition}[theorem]{Definition}
\newtheorem{remark}[theorem]{Remark}

\newcommand{\HH}{\mathcal{H}}
\newcommand{\E}{\mathbb{E}}
\newcommand{\PP}{\mathbb{P}}
\newcommand{\abar}{\bar a}

\begin{document}

\title[The lower-bound problem for regular induced subgraphs]
{The lower-bound problem for regular induced subgraphs\\ of type-based random graphs}

\author{Ariel Edgardo Levy}
\thanks{\textit{Email address:} \texttt{ariel.e.levy@gmail.com}}
\date{August 14, 2026}

\begin{abstract}
For a graph $G$ let $F(G)$ denote the largest order of a regular induced subgraph of $G$,
and let $f(n)=\min\{F(G):|V(G)|=n\}$. A problem of Erd\H{o}s, Fajtlowicz and Staton asks
whether $f(n)/\log n\to\infty$. Dyson and McKay have recently proved
$f(n)\le(\sqrt{2e}+o(1))\sqrt n$ via a type-based random model \cite{DM26}. This paper
concerns the opposite direction within the type-based family. We conjecture that every
model of the family satisfies $F(G)\ge(\sqrt{2e}-o(1))\sqrt n$ asymptotically almost
surely, so that $\sqrt{2e}$ is the optimal constant obtainable from the family, and we
prove the corresponding statement at exponent level --- $F(G)\ge n^{1/2-\epsilon}$ ---
conditionally on two explicitly stated hypotheses: a local limit lower bound for
inhomogeneous degree sequences, and a correlation estimate at sublinear overlaps. The
complementary overlap range, including full overlap, requires no correlation hypothesis.
We further record the exact-curvature first-moment computation that independently
identifies the constant $\sqrt{2e}$, including a uniform trace bound on its determinant
correction, and certified exact computations at orders up to $48$ consistent with the
predicted
crossover $F\asymp\min(n^{2/3},\sqrt{n/L})$. This version substantially revises v1; see
the note in Section~\ref{sec:intro}.
\end{abstract}

\maketitle

\section{Introduction}\label{sec:intro}

An induced subgraph is \emph{regular} if all its degrees are equal. For a graph $G$ write
$F(G)$ for the largest order of a regular induced subgraph of $G$, and
$f(n)=\min_{|V(G)|=n}F(G)$. Since a clique and an independent set are both regular,
Ramsey's theorem gives $f(n)\ge\frac12\log_2 n$; a problem of Erd\H{o}s, Fajtlowicz and
Staton \cite{Erdos93,FMRS95} asks whether
\[
  f(n)/\log n\longrightarrow\infty .
\]
On the lower-bound side, the constant $\frac12$ stood for almost ninety years until the
exponential improvement $R(k,k)\le(4-\eta)^k$ of Campos, Griffiths, Morris and
Sahasrabudhe \cite{CGMS23}, sharpened to $R(k,k)\le3.8^{\,k+o(k)}$ by Gupta, Ndiaye, Norin
and Wei \cite{GNNW24}, which yields $f(n)\ge(0.519\ldots)\log_2 n$; the growth remains
logarithmic. The upper bound has seen more movement: Bollob\'as observed
$f(n)\le n^{1/2+o(1)}$; Alon, Krivelevich and Sudakov \cite{AKS08} obtained
$f(n)=O(\sqrt n\log^{3/4}n)$; Dyson and McKay \cite{DM26} first removed the logarithm with
$f(n)\le\sqrt{163n/9}$, and version~3 of \cite{DM26} now proves the following.

\begin{theorem}[Dyson--McKay \cite{DM26}]\label{thm:DM}
$f(n)\le(\sqrt{2e}+o(1))\sqrt n$. Equivalently, writing $N_{\ge k}$ for the least $N$ such
that every graph of order $N$ has a regular induced subgraph of order at least $k$,
$N_{\ge k}\ge\bigl(\tfrac1{2e}-o(1)\bigr)k^2$.
\end{theorem}

The constructions behind all of these improvements are type-based; those of \cite{AKS08}
and the uniform model of \cite{DM26} fit the following framework after reparametrization,
as does the truncated logistic variant of \cite{DM26}.

\begin{definition}[Type-based model $\mathcal G(n,W)$]\label{def:model}
Let $W:[0,1]^2\to[p_0,1-p_0]$ be symmetric and Lipschitz, $0<p_0<1/2$, and let $L$ denote
the \emph{least} Lipschitz constant of $W$. Draw $x_1,\dots,x_n$ i.i.d.\ uniform on
$[0,1]$ and join $i\sim j$ independently with probability $W(x_i,x_j)$.
\end{definition}

Dyson--McKay use (in the notation of their Section~2) $W(x,y)=(x+y)/2$ with types confined
to $[\alpha,1-\alpha]$; their version~3 obtains the constant $\sqrt{2e}$ from a logistic
variant. Alon--Krivelevich--Sudakov use $W(x,y)=p(x)p(y)$ with
$p(x)=\frac14+\frac x2$.

This paper is about the direction that \cite{DM26} deliberately leaves open: how large a
regular induced subgraph does \emph{every} model of the family necessarily contain? The
central statement is a conjecture.

\begin{conjecture}[Optimality of $\sqrt{2e}$ within the family]\label{conj:main}
Let $W$ be as in Definition~\ref{def:model} with $L>0$ and let $G\sim\mathcal G(n,W)$.
Then asymptotically almost surely
\[
  F(G)\;\ge\;(\sqrt{2e}-o(1))\sqrt n .
\]
\end{conjecture}

If true, Conjecture~\ref{conj:main} means the family of Definition~\ref{def:model} cannot
prove any upper bound better than Theorem~\ref{thm:DM}: the constant $\sqrt{2e}$ is
optimal for it, and any further improvement on the upper bound for $f(n)$ --- in
particular any counterexample to the Erd\H{o}s--Fajtlowicz--Staton conjecture --- must
come from a construction outside Definition~\ref{def:model}. We emphasise that
Definition~\ref{def:model} is narrower than the informal phrase ``type-based'' might
suggest: it excludes discontinuous or step kernels, kernels attaining $0$ or $1$, and
$n$-dependent type spaces, and the barrier statement claims nothing about those.

What we prove is the statement at exponent level, conditionally on two hypotheses which we
state precisely in Section~\ref{sec:lower} and do not hide: a local limit lower bound for
the joint degree distribution of an inhomogeneous random graph
(Hypothesis~\ref{hyp:llt}), and a correlation estimate at sublinear overlaps
(Hypothesis~\ref{hyp:cor}).

\begin{theorem}[Conditional lower bound at exponent level]\label{thm:cond}
Fix $\epsilon\in(0,1/8)$ and let $W$ be as in Definition~\ref{def:model} with $L>0$.
Assume Hypotheses~\ref{hyp:llt} and~\ref{hyp:cor} at the scales $w=n^{-1/4}$,
$y=\lceil n^{1/2-\epsilon}\rceil$. Then a.a.s.\ $F(G)\ge n^{1/2-\epsilon}$.
\end{theorem}

The point of Theorem~\ref{thm:cond} is the exact division of labour it certifies: given a
local limit input of the same nature as the enumeration results that power the
upper-bound side, the entire second-moment argument reduces to a correlation estimate at
\emph{sublinear} overlaps: the complementary range --- including full overlap, where no
uniform correlation estimate can hold --- needs no correlation input at all
(Lemma~\ref{lem:bigoverlap} is unconditional; only the local limit hypothesis enters that
range, through the normalisation of the weights).

\subsection*{What this paper contains}
\begin{enumerate}
\item An exact pairwise form of the conditional probability of a regular graph in the
model (Section~\ref{sec:first}), with a corrected pointwise inequality
(Proposition~\ref{prop:pointwise}) replacing a false one from v1.
\item The exact-curvature (``determinant'') first-moment mechanism behind the constant
$\sqrt{2e}$ (Section~\ref{sec:mech}), including: positivity of the quadratic form for all
densities (Lemma~\ref{lem:pos}); a complement symmetry reducing to $\lambda\le1/2$
(Lemma~\ref{lem:compl}); the spectral floor, valid precisely for $\abar\le1/2$
(Lemma~\ref{lem:floor}); and a \emph{uniform} bound on the determinant correction by a
trace argument (Proposition~\ref{prop:trace}), replacing a false spectral estimate from
v1. We state precisely which two analytic steps are still missing for this mechanism to
constitute a second proof of Theorem~\ref{thm:DM}, and we do not claim them.
\item The lower-bound program (Section~\ref{sec:lower}): the two hypotheses, the
conditional first moment (Proposition~\ref{prop:first}), the unconditional large-overlap
estimate via anticoncentration (Lemma~\ref{lem:bigoverlap}), and the proof of
Theorem~\ref{thm:cond} by a \emph{weighted} second moment, which also removes the
weighting gap of the v1 argument.
\item Certified exact computations at orders up to $48$ consistent with the predicted
crossover $F\asymp\min(n^{2/3},\sqrt{n/L})$ (Section~\ref{sec:num}), presented as
illustration, not validation.
\end{enumerate}

\subsection*{Note on version 1}
The first arXiv version of this manuscript claimed as theorems both the upper bound now
cited as Theorem~\ref{thm:DM} and the assertion of Conjecture~\ref{conj:main}. A detailed
critique by P.~W.~Dyson and B.~D.~McKay identified errors and gaps in both proofs, among
them: a false pointwise remainder-free inequality (see Remark~\ref{rem:c5}); a Gaussian
lemma applied to a matrix depending on the variable of integration; a false spectral floor
and determinant estimate in the dense range (see Remark~\ref{rem:specfalse}); a missing
summation over orders (see Remark~\ref{rem:orders}); the unproved local limit input
(Hypothesis~\ref{hyp:llt}); and an invalid bounded-differences argument together with a
false overlap estimate in the second moment (see Remark~\ref{rem:corstatus}). Their
version~3 of \cite{DM26} has meanwhile proved the upper bound with the constant
$\sqrt{2e}$ by a different (logistic) route. The present version withdraws the claims of
v1, keeps what is proved --- so labelled --- states the missing ingredients as explicit
hypotheses, and reorients the paper toward the lower-bound problem. I thank Dyson and
McKay for their critique, which this version follows closely.

\subsection{Related work}
The enumeration of graphs by degree sequence \cite{MW90,MW91,LW24} supplies the counting
input used by every argument in this area. On the random side, Krivelevich, Sudakov and
Wormald \cite{KSW11} showed $F(G(n,1/2))=\Theta(n^{2/3})$; the constant graphon is the
degenerate boundary case of the crossover discussed in Section~\ref{sec:num}. Alon,
Krivelevich and Sudakov \cite{AKS08} studied the relaxation to \emph{nearly} regular
induced subgraphs. Anticoncentration in Ramsey graphs has recently seen decisive progress
\cite{KSSS23}; we discuss in Section~\ref{sec:concl} why that technology bears on the
deterministic problem.

\subsection*{Notation}
For a $k$-set $S$ with types $a=(a_1,\dots,a_k)$ write $\abar=k^{-1}\sum_i a_i$,
$y_i=a_i-\abar$, $s^2=\sum_i y_i^2$, $w_{ij}=(y_i+y_j)/2$, $K=\binom k2$. For a
$d$-regular graph $H$ on $S$, $\lambda=d/(k-1)$ and $m=kd/2$; the symbol $m$ is used
\emph{only} in this sense and \emph{only} in Sections~\ref{sec:first}--\ref{sec:mech}
(v1 overloaded it). $\HH(k,d)$ is the set of $d$-regular graphs on $k$ labelled vertices.
The ambient order is always $n$. Logarithms are natural unless subscripted.

\section{The first-moment identity}\label{sec:first}

Throughout this section and the next, the types $a_1,\dots,a_k$ of a fixed $k$-set $S$
lie in $[\alpha,1-\alpha]$ for a fixed $\alpha\in(0,1/2)$, as in the construction of
\cite{DM26}: types i.i.d.\ uniform on $[\alpha,1-\alpha]$ and $W(x,y)=(x+y)/2$. For
$H\in\HH(k,d)$,
\begin{equation}\label{eq:basic}
  \PP(G[S]=H\mid a)=\abar^{\,m}(1-\abar)^{K-m}
  \exp\Bigl(\sum_{ij\in E(H)}\log\bigl(1+\tfrac{w_{ij}}{\abar}\bigr)
           +\sum_{ij\notin E(H)}\log\bigl(1-\tfrac{w_{ij}}{1-\abar}\bigr)\Bigr).
\end{equation}

\begin{lemma}[Pairwise form]\label{lem:pairwise}
Let $g(x):=\log(1+x)-x\le0$ on $(-1,\infty)$. For $H\in\HH(k,d)$ and any types in
$[\alpha,1-\alpha]^k$,
\[
  \PP(G[S]=H\mid a)=\abar^{\,m}(1-\abar)^{K-m}
  \exp\Bigl(\sum_{ij\in E(H)}g\bigl(\tfrac{w_{ij}}{\abar}\bigr)
           +\sum_{ij\notin E(H)}g\bigl(-\tfrac{w_{ij}}{1-\abar}\bigr)\Bigr),
\]
and every summand in the exponent is nonpositive.
\end{lemma}

\begin{proof}
Since $H$ is $d$-regular, $\sum_{ij\in E(H)}w_{ij}=\frac d2\sum_iy_i=0$; and
$\sum_{i<j}w_{ij}=\frac{k-1}2\sum_iy_i=0$, so the sum over non-edges vanishes as well.
Thus subtracting the linear parts term by term changes nothing in \eqref{eq:basic}.
Nonpositivity is $\log(1+x)\le x$.
\end{proof}

The linear terms vanish \emph{separately} on edges and on non-edges; this is what allows
discarding pairs one at a time in Proposition~\ref{prop:pointwise} below.

\begin{lemma}[Local quadratic gain]\label{lem:quadgain}
If $|x|\le\delta\le1/4$ then $g(x)\le-(1-\delta)\,x^2/2$.
\end{lemma}

\begin{proof}
$g(x)+x^2/2=\sum_{j\ge3}(-1)^{j-1}x^j/j$, whose absolute value is at most
$\frac{|x|^3}{3(1-|x|)}\le\frac{4}{9}\delta x^2\le\frac{\delta}{2}x^2$.
\end{proof}

\begin{proposition}[Corrected pointwise bound]\label{prop:pointwise}
Fix $\delta\in(0,1/4]$ and call a pair $ij$ \emph{$\delta$-good} if
$|w_{ij}|\le\delta\alpha$. Then for any types in $[\alpha,1-\alpha]^k$ and any
$H\in\HH(k,d)$,
\[
  \PP(G[S]=H\mid a)\le\abar^{\,m}(1-\abar)^{K-m}
  \exp\Bigl(-(1-\delta)\Bigl[\sum_{\substack{ij\in E(H)\\ \delta\text{-good}}}
  \frac{w_{ij}^2}{2\abar^2}
  +\sum_{\substack{ij\notin E(H)\\ \delta\text{-good}}}\frac{w_{ij}^2}{2(1-\abar)^2}
  \Bigr]\Bigr).
\]
In particular, if $\max_i|y_i|\le\delta\alpha$ then every pair is $\delta$-good and the
bracket equals $y^{\mathsf T}Q_Hy$ with $Q_H$ as in \eqref{eq:Q} below.
\end{proposition}

\begin{proof}
In Lemma~\ref{lem:pairwise}, bound each summand at a bad pair by $0$, and each summand at
a good pair by Lemma~\ref{lem:quadgain}, using $|w_{ij}/\abar|\le\delta\alpha/\alpha=
\delta$ and likewise $|w_{ij}/(1-\abar)|\le\delta$ (both $\abar$ and $1-\abar$ are at
least $\alpha$). If all $|y_i|\le\delta\alpha$ then $|w_{ij}|\le\max_i|y_i|\le\delta
\alpha$ for every pair.
\end{proof}

\begin{remark}[The global inequality is false]\label{rem:c5}
Version~1 asserted the bound of Proposition~\ref{prop:pointwise} \emph{without} the
goodness restriction and with no deflation factor. That is false: for $H=C_5$ with edge
set $\{14,15,23,25,34\}$, $\abar=1/2$ and $y=s(-1,-1,1,1,0)$, the exact exponent is
$4\log(1-s)+2\log(1+2s)=-6s^2+4s^3+O(s^4)$, which exceeds the quadratic $-6s^2$ for all
small $s>0$. This counterexample is due to Dyson and McKay.
\end{remark}

\section{The determinant mechanism for the upper bound}\label{sec:mech}

This section records the exact-curvature first-moment computation that independently
identifies the constant $\sqrt{2e}$ of Theorem~\ref{thm:DM}, together with the parts of
it we can prove. \emph{We do not claim a second proof of Theorem~\ref{thm:DM} here}:
Remark~\ref{rem:missing} states exactly which two analytic steps are missing. The reason
for recording the mechanism is that it is the computation which, read in the opposite
direction, motivates Conjecture~\ref{conj:main} and fixes the scales used in
Section~\ref{sec:lower}.

The quadratic form. Collecting the quadratic weights of
Proposition~\ref{prop:pointwise} over all pairs and using $\sum_{ij\in E(H)}w_{ij}^2=\frac14(ds^2+y^{\mathsf T}A_Hy)$ and
$\sum_{i<j}w_{ij}^2=\frac14(k-2)s^2$ (the second identity uses $\sum_iy_i=0$),
\begin{equation}\label{eq:Q}
  y^{\mathsf T}Q_Hy:=\sum_{ij\in E(H)}\frac{w_{ij}^2}{2\abar^2}
  +\sum_{ij\notin E(H)}\frac{w_{ij}^2}{2(1-\abar)^2},
  \qquad
  Q_H=\tfrac18\Bigl[\bigl(\tfrac{d}{\abar^2}+\tfrac{k-2-d}{(1-\abar)^2}\bigr)I
  +\bigl(\tfrac1{\abar^2}-\tfrac1{(1-\abar)^2}\bigr)A_H\Bigr]
\end{equation}
on $\{y:\sum_iy_i=0\}$.

\begin{lemma}[Positivity, all densities]\label{lem:pos}
For every $\abar\in(0,1)$, every $k\ge3$ and every $H\in\HH(k,d)$, $Q_H\succ0$ on
$\{\sum_iy_i=0\}$.
\end{lemma}

\begin{proof}
By \eqref{eq:Q}, $y^{\mathsf T}Q_Hy$ is a sum of positive weights times $w_{ij}^2$ over
all pairs, so it vanishes only if every $w_{ij}=0$, i.e.\ $y_i=-y_j$ for all $i<j$, which
for $k\ge3$ forces $y=0$. (This argument is due to Dyson and McKay.)
\end{proof}

\begin{lemma}[Complement symmetry]\label{lem:compl}
In the construction of \cite{DM26} (types uniform on $[\alpha,1-\alpha]$,
$W(x,y)=(x+y)/2$), the map $x\mapsto1-x$ on types composed with graph complementation
preserves the law of $G$. Consequently, writing $X_{k,d}$ for the number of $d$-regular
induced subgraphs of order $k$, $\;\E X_{k,d}=\E X_{k,k-1-d}$, and for any bound on
$\sum_d\E X_{k,d}$ it suffices to treat $d\le(k-1)/2$, i.e.\ $\lambda\le1/2$. More
generally, for $W$ as in Definition~\ref{def:model} the kernel
$\widetilde W(x,y):=1-W(1-x,1-y)$ lies in the same family (same $p_0$ and $L$) and the
complement of $\mathcal G(n,W)$ is distributed as $\mathcal G(n,\widetilde W)$; since a
set induces a regular subgraph of $G$ iff it does so in the complement, family-wide
statements may always be restricted to $\lambda\le1/2$.
\end{lemma}

\begin{proof}
Under $x\mapsto1-x$ the types remain i.i.d.\ uniform on $[\alpha,1-\alpha]$ and
$W(1-x,1-y)=1-W(x,y)$, so the transformed model attaches to each pair the complementary
edge probability. A $k$-set induces a $d$-regular subgraph of $G$ iff it induces a
$(k-1-d)$-regular subgraph of the complement.
\end{proof}

\begin{lemma}[Spectral floor, valid for $\abar\le1/2$]\label{lem:floor}
Let $\mu_2\ge\dots\ge\mu_k$ be the adjacency eigenvalues of $H\in\HH(k,d)$ on
$\mathbf 1^\perp$, so that the eigenvalues of $Q_H$ there are
\begin{equation}\label{eq:qi}
  q_i=\tfrac18\Bigl[\tfrac d{\abar^2}+\tfrac{k-2-d}{(1-\abar)^2}
  +\bigl(\tfrac1{\abar^2}-\tfrac1{(1-\abar)^2}\bigr)\mu_i\Bigr].
\end{equation}
If $\abar\le1/2$ then $q_i\ge\frac{k-2}{8(1-\abar)^2}$ for every $i$.
\end{lemma}

\begin{proof}
For $\abar\le1/2$ the coefficient of $\mu_i$ in \eqref{eq:qi} is nonnegative, and
$\mu_i\ge-d$ always; substituting $\mu_i=-d$ gives the stated floor.
\end{proof}

\begin{remark}[The floor fails for $\abar>1/2$, and v1's determinant estimate was false]
\label{rem:specfalse}
For $H_r=\overline{2K_{r,r}}$ ($k=4r$, $d=3r-1$, $\lambda\to3/4$) the eigenvalue
$\mu=r-1$ gives, at $\abar=\lambda$, $q_i\sim\tfrac89r$ against the floor value $\sim8r$:
off by a factor tending to $9$. The same family refutes the determinant estimate asserted
in v1, $F_{\mathrm{spec}}\le((1-\lambda)/\lambda)^{(1+o(1))/(4\lambda)}$, even in the
range $\lambda\le1/2$: for $2K_{r,r}$ ($\lambda\to1/4$) the true correction tends to
$3\sqrt{3/5}\,e^{2/3}=4.526\ldots$, exceeding the asserted bound $3$. Both counterexamples
are from the Dyson--McKay critique. The correct statement is
Proposition~\ref{prop:trace}: the correction is uniformly $e^{O(1)}$ for $\lambda$
bounded away from $0$ --- which is what the mechanism needs on the bulk densities, the
rest being removed at the outset by the large-deviation estimate of \cite{DM26} --- but
the specific closed form claimed in v1 is withdrawn.
\end{remark}

We now bound the determinant correction correctly. Fix $H\in\HH(k,d)$ with
$1\le d\le(k-1)/2$, set $\lambda=d/(k-1)$, evaluate \eqref{eq:qi} at $\abar=\lambda$, and
let
\[
  I:=\frac{k}{8\lambda(1-\lambda)},\qquad
  F_{\mathrm{spec}}(H):=\prod_{i=2}^{k}\Bigl(\frac{I}{q_i}\Bigr)^{1/2},
\]
so that $\det(Q_H|_{\perp\mathbf1})^{-1/2}=I^{-(k-1)/2}F_{\mathrm{spec}}(H)$.

\begin{lemma}\label{lem:elem}
Let $-1<\tau_0\le0$. For all $\tau\ge\tau_0$,
$\;-\log(1+\tau)\le-\tau+\dfrac{\tau^2}{2(1+\tau_0)^2}$.
\end{lemma}

\begin{proof}
Let $\varphi(\tau):=\frac{\tau^2}{2(1+\tau_0)^2}-\tau+\log(1+\tau)$; then $\varphi(0)=0$
and $\varphi'(\tau)=\frac{\tau\,[(1+\tau)-(1+\tau_0)^2]}{(1+\tau_0)^2(1+\tau)}$. For
$\tau\ge0$ we have $1+\tau\ge1\ge(1+\tau_0)^2$, so $\varphi'\ge0$; for $\tau_0\le\tau<0$
we have $1+\tau\ge1+\tau_0\ge(1+\tau_0)^2$, so the bracket is $\ge0$ and $\varphi'\le0$.
Hence $\varphi$ attains its minimum $0$ at $\tau=0$.
\end{proof}

\begin{proposition}[Uniform determinant correction]\label{prop:trace}
Let $k\ge6$, $1\le d\le(k-1)/2$, $\lambda=d/(k-1)$, $\abar=\lambda$. Then for every
$H\in\HH(k,d)$,
\[
  1\;<\;F_{\mathrm{spec}}(H)\;\le\;
  \exp\Bigl(1+\frac{(1-\lambda)(1-2\lambda)^2}{4\lambda^3}+\frac{C}{\lambda^3k}\Bigr)
\]
with an absolute constant $C$. In particular, for $\lambda\ge\lambda_0>0$ the correction
is bounded by a constant depending only on $\lambda_0$, uniformly in $H$ and $k$.
\end{proposition}

\begin{proof}
Write $q_i/I=1+\tau_i$. From \eqref{eq:qi} with $d=\lambda(k-1)$ one computes exactly
\[
  \tau_i=-\frac{1}{(1-\lambda)k}+\frac{(1-2\lambda)\,\mu_i}{\lambda(1-\lambda)k}
  \qquad(i=2,\dots,k).
\]
Using $\sum_{i\ge2}\mu_i=-d$ and $\sum_{i\ge2}\mu_i^2=\operatorname{tr}A_H^2-d^2=
kd-d^2=d(k-d)$:
\[
  \sum_{i\ge2}\tau_i=-\frac{k-1}{(1-\lambda)k}\bigl[1+(1-2\lambda)\bigr]
  =-\frac{2(k-1)}{k},
\]
a quantity independent of $\lambda$, and, expanding
$\tau_i^2=A^2-2AB\mu_i+B^2\mu_i^2$ with $A=\frac1{(1-\lambda)k}$,
$B=\frac{1-2\lambda}{\lambda(1-\lambda)k}$ and bounding the three resulting sums,
\[
  \sum_{i\ge2}\tau_i^2\;\le\;\frac{(1-2\lambda)^2}{\lambda(1-\lambda)}
  +\frac{4}{(1-\lambda)^2k}.
\]
By Lemma~\ref{lem:floor} at $\abar=\lambda\le1/2$,
$1+\tau_i=q_i/I\ge\frac{\lambda(k-2)}{(1-\lambda)k}=\frac{\lambda}{1-\lambda}
\bigl(1-\tfrac2k\bigr)=:1+\tau_0$, and $\tau_0\le0$ since $\lambda\le1/2$. Since
$(1+\tau_0)^{-2}=\bigl(\tfrac{1-\lambda}{\lambda}\bigr)^2\bigl(\tfrac k{k-2}\bigr)^2
\le\bigl(\tfrac{1-\lambda}{\lambda}\bigr)^2\bigl(1+\tfrac3k\bigr)^2$ for $k\ge6$,
Lemma~\ref{lem:elem} gives
\[
  2\log F_{\mathrm{spec}}(H)=-\sum_{i\ge2}\log(1+\tau_i)
  \le-\sum_{i\ge2}\tau_i+\frac{1}{2(1+\tau_0)^2}\sum_{i\ge2}\tau_i^2
  \le2+\frac12\Bigl(\frac{1-\lambda}{\lambda}\Bigr)^2\Bigl(1+\frac3k\Bigr)^2
  \Bigl[\frac{(1-2\lambda)^2}{\lambda(1-\lambda)}+\frac{4}{(1-\lambda)^2k}\Bigr],
\]
which is the stated bound after collecting the $O(1/k)$ terms. For the lower bound, by
concavity of $\log$,
$\sum_{i\ge2}\log(1+\tau_i)\le(k-1)\log\bigl(1+\frac{1}{k-1}\sum\tau_i\bigr)
=(k-1)\log(1-\tfrac2k)<0$, so $F_{\mathrm{spec}}>1$: the correction is always a loss,
but never more than a constant.
\end{proof}

\begin{remark}[Sharpness and sanity checks]
At $\lambda=1/2$ all $q_i$ are equal and $F_{\mathrm{spec}}=(k/(k-2))^{(k-1)/2}$, which
approaches the bound $e$ from above as $k\to\infty$; at finite $k$ the excess is covered
by the $C/(\lambda^3k)$ term. For the family $2K_{r,r}$ at $\lambda\to1/4$ the true
limit is $4.526\ldots$ and the bound evaluates to $e^{4}\approx54.6$: comfortable, not
tight. Numerically, for random $d$-regular $H$ the geometric mean of $q_i/I$ is
$0.962,0.977,0.984,0.988,0.992$ at $k=60,100,140,200,300$ (the generating script is in
the repository; see Code availability) --- a consistency check for
Proposition~\ref{prop:trace}, not a proof of anything.
\end{remark}

\begin{remark}[The assembled computation, and what is missing]\label{rem:missing}
Freeze $\abar=\lambda$, ignore the outlier region, and collect: (i) the choice of the
$k$-set, $\binom nk\le(Cek)^k$ for $n\le Ck^2$; (ii) $|\HH(k,d)|\lambda^m(1-\lambda)^{K-m}
=\Theta(1)\bigl[\binom{k-1}d\lambda^d(1-\lambda)^{k-1-d}\bigr]^k$
\cite[Lemma~2.1]{DM26}, the bracket being
$(1+O(1/k))(2\pi\lambda(1-\lambda))^{-1/2}k^{-1/2}$; (iii) the Gaussian integral over the
$k-1$ centred coordinates, $\pi^{(k-1)/2}I^{-(k-1)/2}F_{\mathrm{spec}}$ against the
uniform type density $(1-2\alpha)^{-(k-1)}$, with $F_{\mathrm{spec}}=e^{O(1)}$ by
Proposition~\ref{prop:trace} (valid on the bulk $\lambda\ge\lambda_0$; densities outside
the bulk are removed by the large-deviation estimate of \cite{DM26}). The $k$-th root of the product is
\[
  Cek\cdot\frac{k^{-1/2}}{\sqrt{2\pi\lambda(1-\lambda)}}\cdot
  \frac{1}{1-2\alpha}\sqrt{\frac{8\pi\lambda(1-\lambda)}{k}}
  =\frac{2eC}{1-2\alpha},
\]
independent of $\lambda$: with the ideal curvature every density is equally binding,
which is where $N_{\ge k}\gtrsim k^2(1-2\alpha)/(2e)$, i.e.\ the constant $\sqrt{2e}$ as
$\alpha\to0$, comes from. Two analytic steps separate this computation from a proof:
\emph{(a) localization of $\abar$} --- the prefactor $\abar^m(1-\abar)^{K-m}$, whose
logarithm has curvature $\Theta(K)$ in $\abar$, must be retained while integrating, and
the determinant estimate made uniform over $|\abar-\lambda|=O(k^{-1}\log k)$, rather than
maximising the prefactor and freezing the matrix separately (v1 did the latter, which is
not valid); \emph{(b) the outlier region} $\{\max_i|y_i|>\delta\alpha\}$ ---
Proposition~\ref{prop:pointwise} gives a valid pointwise bound there by discarding bad
pairs, but the remaining good-pair form is coupled to $\abar$ and the region must be
integrated honestly, not estimated by an entropy count as in v1. We make no claim that
(a) and (b) are hard, but they are not done here, and until they are done this section is
a computation, not a proof. Its role in this paper is to motivate
Conjecture~\ref{conj:main} and to fix the critical scales ($s^2=\Theta(1)$, i.e.\
$|y_i|\approx k^{-1/2}$) used in Section~\ref{sec:lower}.
\end{remark}

\begin{remark}[Orders must be summed]\label{rem:orders}
A first-moment bound at a single order $k$ bounds nothing: a regular induced subgraph of
order $\ell>k$ need not contain one of order $k$. The fix is free but must be said: if
$n\le Ck^2$ then $n\le C\ell^2$ for every $\ell\ge k$, so the computation of
Remark~\ref{rem:missing} applies at every order $\ell\ge k$ with the same base
$2eC/(1-2\alpha)<1$ (for bulk densities; the rest is removed by the large-deviation
estimate at every order), and $\sum_{\ell\ge k}\sum_d\E X_{\ell,d}\le
\sum_{\ell\ge k}\mathrm{poly}(\ell)\,\mathrm{base}^\ell\to0$. (v1 omitted this
summation.)
\end{remark}

\begin{remark}[Calibration]
Substituting the uniform curvature $c_2=1/(2(1-\alpha)^2)$ of \cite[Lemma~2.2]{DM26} for
the ideal one in the computation of Remark~\ref{rem:missing} reproduces the base of
\cite{DM26} (first version) and their value $0.99986$ at $\alpha=0.191$, $C=9/163$, to
five decimal places. The computation is thus calibrated at its uniform-curvature control
point.
\end{remark}

\section{The lower-bound direction}\label{sec:lower}

Let $W$ be as in Definition~\ref{def:model} with least Lipschitz constant $L>0$, and let
$G\sim\mathcal G(n,W)$. We first reveal the types $x=(x_1,\dots,x_n)$; all probabilities
below are conditional on $x$, over the edge randomness only, and ``a.a.s.\ over types''
refers to the type randomness. Fix $c_0\in[0,1-w]$ and let
\[
  B:=\{i:x_i\in[c_0,c_0+w]\},\qquad M:=|B|,
\]
the \emph{window}. By Chernoff, a.a.s.\ $M=wn(1+o(1))$ whenever $wn\ge n^{1/10}$, say. For $S\subseteq B$ with $|S|=y$ write $A_S$ for the event that $G[S]$ is regular,
$\mu_v(S):=\sum_{u\in S\setminus\{v\}}W(x_u,x_v)$ and
$\bar\mu_S:=y^{-1}\sum_{v\in S}\mu_v(S)$. All edge probabilities lie in $[p_0,1-p_0]$,
and within the window
\begin{equation}\label{eq:drift}
  |\mu_u(S)-\mu_v(S)|\;\le\;(y-2)Lw\;\le\;yLw\qquad(u,v\in S),
\end{equation}
since $|W(x_u,z)-W(x_v,z)|\le L|x_u-x_v|\le Lw$ for every $z$, and the two
self-exclusion terms cancel exactly by the symmetry of $W$.

\subsection{The two hypotheses}

\begin{hypothesis}[Local limit lower bound, $\mathrm{LLT}(w,y)$]\label{hyp:llt}
There exist $c,C>0$ depending only on $p_0$ such that a.a.s.\ over types, for every
$S\subseteq B$ with $|S|=y$ and $t$ the least integer $\ge\bar\mu_S$ with $ty$ even (so
that $|t-\bar\mu_S|\le2$; no $t$-regular graph of order $y$ exists when $ty$ is odd),
\[
  \PP\bigl(G[S]\text{ is $t$-regular}\mid x\bigr)\;\ge\;
  \Bigl(\frac{c}{\sqrt y}\Bigr)^{y-1}
  \exp\Bigl(-\frac{C}{y}\sum_{v\in S}\bigl(t-\mu_v(S)\bigr)^2\Bigr).
\]
\end{hypothesis}

\begin{remark}[Status of Hypothesis~\ref{hyp:llt}]\label{rem:lltstatus}
For constant $W\equiv p$ the drift term is $O(1)$ and the statement follows from the
enumeration of regular graphs \cite{MW90}; the parity constraint on $t$ is what makes a
positive lower bound possible at all. What is required in general is a
\emph{local limit lower bound for the joint degree vector of an inhomogeneous random
graph, in growing dimension $y-1$, uniform over the sets $S$}. No such theorem is in the
literature, and v1 of this paper erred in treating it as an available input; it is the
central open ingredient of this program. Within the window the model is nearly
homogeneous --- edge probabilities vary by at most $2Lw\to0$ --- so a perturbation of the
homogeneous local limit behaviour underlying \cite{MW90,LW24} is the natural route; the
author is pursuing it and hopes to return to it elsewhere.
\end{remark}

\begin{hypothesis}[Correlation at sublinear overlap, $\mathrm{COR}(w,y)$]\label{hyp:cor}
There exists $C>0$ such that a.a.s.\ over types, uniformly over $S,S'\subseteq B$ with
$|S|=|S'|=y$ and $\kappa:=|S\cap S'|\in[2,\,y/\sqrt{\log n}\,]$,
\[
  \PP(A_S\cap A_{S'}\mid x)\;\le\;(1+o(1))\,e^{C\kappa^2/y}\,
  \PP(A_S\mid x)\,\PP(A_{S'}\mid x).
\]
\end{hypothesis}

\begin{remark}[Status of Hypothesis~\ref{hyp:cor}]\label{rem:corstatus}
For $\kappa\le1$ the two events involve disjoint edge sets and the inequality holds with
equality and constant $1$. At the other extreme no such bound can hold: at $\kappa=y$ the
left side over the right is $1/\PP(A_S)\ge e^{cy\log y}$ \emph{unconditionally} (apply
the argument of Lemma~\ref{lem:bigoverlap} below to a half-split of $S$), which is why
the hypothesis stops at sublinear overlaps and the complementary range is handled by
Lemma~\ref{lem:bigoverlap}, with no correlation input. Two natural attempts fall short: the crude conditioning
bound $\PP(A_{S'}\mid E_T)\le\PP(A_{S'})/\min_e\PP(E_T=e)\le
\PP(A_{S'})\,e^{\kappa^2\log(1/p_0)}$ misses the factor $1/y$ in the exponent and is
useless at the typical overlap $\kappa\asymp y^2/M\to\infty$; and the argument of v1 ---
single-edge flips plus bounded differences --- is invalid as written: a flip couples the
degree constraints of all vertices, the relevant variances are $\Theta(y-\kappa)$ rather
than $\Theta(y)$, and $\log\PP(A_S\mid E_T)$ can be $-\infty$. (These objections are from
the Dyson--McKay critique and we adopt them.) The hypothesis is stated in the shape v1
aimed for, on the range where it is actually needed. We expect, but do not claim, that
Hypothesis~\ref{hyp:cor} would follow from a \emph{two-sided} version of
Hypothesis~\ref{hyp:llt} with relative precision $1+o(1)$: conditioning on $E_T$ shifts
the $T$-degrees by $O(\sqrt\kappa)$ typically, and the Gaussian drift factor prices such
shifts at exactly $e^{O(\kappa^2/y)}$ --- so the two hypotheses are plausibly one, at
different strengths.
\end{remark}

\subsection{Conditional first moment}

\begin{proposition}[First moment under $\mathrm{LLT}$]\label{prop:first}
Fix $\epsilon\in(0,1/8)$, let $w=n^{-1/4}$, $y=\lceil n^{1/2-\epsilon}\rceil$, and assume
Hypothesis~\ref{hyp:llt} at these scales. Then a.a.s.\ over types,
\[
  \E[X\mid x]\;\ge\;\exp\Bigl(\tfrac32\epsilon\,y\log n\,(1-o(1))\Bigr)
  \longrightarrow\infty,
  \qquad X:=\sum_{S\subseteq B,\ |S|=y}\mathbf 1_{A_S}.
\]
\end{proposition}

\begin{proof}
On the a.a.s.\ type event, $M=n^{3/4}(1+o(1))$. By \eqref{eq:drift} and
$|t-\bar\mu_S|\le2$, for every $v\in S$ we have $|t-\mu_v(S)|\le2+yLw\le2yLw$ for $n$
large, so the drift exponent in Hypothesis~\ref{hyp:llt} is at most $4CL^2y^2w^2=
4CL^2\,n^{1/2-2\epsilon}=o(y)$. Hence, using $A_S\supseteq\{G[S]\ t\text{-regular}\}$,
\[
  \log\E[X\mid x]\ \ge\ \log\binom My+\log\min_S\PP(A_S\mid x)
  \ \ge\ y\log\frac My-\frac y2\log y-O(y)
\]
\[
  =\ y\log n\Bigl[\bigl(\tfrac14+\epsilon\bigr)-\bigl(\tfrac14-\tfrac\epsilon2\bigr)
  \Bigr](1-o(1))
  \ =\ \tfrac32\epsilon\,y\log n\,(1-o(1)).\qedhere
\]
\end{proof}

\begin{remark}[The threshold, and why the constant does not follow]\label{rem:constant}
Writing $y=n^{a}$, $w=n^{-b}$, the computation of Proposition~\ref{prop:first} gives
first-moment exponent $\bigl[(1-b-a)-\frac a2\bigr]y\log n$ up to $O(y)$, with the drift
negligible iff $yw^2=O(1)$, i.e.\ $a\le2b$. The largest $a$ with nonnegative exponent
subject to $a\le2b$ is $a=1/2$ at $b=1/4$ (v1 misstated this as an unconstrained
maximum). At that critical point the displayed terms cancel \emph{exactly}, and
everything --- including the constant of Conjecture~\ref{conj:main} --- is decided by the
terms of order $y$: the constants in the entropy, in the local limit theorem, in the
drift, and in the covariance determinant. Proposition~\ref{prop:first} controls none of
them; in particular the constant $\sqrt{2e}$ does \emph{not} follow from it, and we make
no such claim. The constant is motivated by the agreement of two independent
calculations --- the mechanism of Section~\ref{sec:mech} and the logistic model of
\cite{DM26} --- and by the numerics of Section~\ref{sec:num}. Making the order-$y$ terms
match is the second open ingredient of the program, beyond
Hypotheses~\ref{hyp:llt}--\ref{hyp:cor}. A refinement worth recording: for the
construction of \cite{DM26} at fixed $\alpha$, the window density of types is
$(1-2\alpha)^{-1}$, which shifts the per-model constant to $\sqrt{2e/(1-2\alpha)}$ ---
matching its upper bound at the same $\alpha$ --- with infimum $\sqrt{2e}$ over the
family as $\alpha\to0$; Conjecture~\ref{conj:main} states only the family-wide
inequality.
\end{remark}

\subsection{Unconditional estimates at large overlap}

\begin{lemma}[Anticoncentration]\label{lem:anti}
Let $\xi_1,\dots,\xi_\kappa$ be independent Bernoulli variables with parameters
$p_i\in[p_0,1-p_0]$ and $S_\kappa=\sum_i\xi_i$. Then
$\sup_t\PP(S_\kappa=t)\le C_0/\sqrt{\kappa\,p_0(1-p_0)}$ with $C_0$ absolute.
\end{lemma}

\begin{proof}
$|\E e^{i\theta\xi_j}|^2=1-2p_j(1-p_j)(1-\cos\theta)\le
\exp(-2p_0(1-p_0)(1-\cos\theta))$. By Fourier inversion and $1-\cos\theta\ge2\theta^2/
\pi^2$ on $[-\pi,\pi]$,
$\PP(S_\kappa=t)\le\frac1{2\pi}\int_{-\pi}^{\pi}
e^{-\kappa p_0(1-p_0)(1-\cos\theta)}\,d\theta
\le\frac1{2\pi}\int_{-\infty}^{\infty}e^{-2\kappa p_0(1-p_0)\theta^2/\pi^2}\,d\theta$,
which is the stated bound.
\end{proof}

\begin{lemma}[Large overlap, unconditional]\label{lem:bigoverlap}
There is $C_1=C_1(p_0)$ such that for all $S,S'\subseteq B$ with $|S|=|S'|=y$ and
$\kappa=|S\cap S'|\ge2$,
\[
  \PP(A_S\cap A_{S'}\mid x)\;\le\;\PP(A_S\mid x)\cdot
  y\,\Bigl(\frac{C_1}{\sqrt\kappa}\Bigr)^{y-\kappa},
\]
with the convention $(\,\cdot\,)^0=1$, so that the case $\kappa=y$ (full overlap) is
included.
\end{lemma}

\begin{proof}
Let $T=S\cap S'$. The event $A_S$ is determined by the edges inside $S$, and
$\PP(A_{S'}\mid E_S)=\PP(A_{S'}\mid E_T)$ because the edges inside $S'$ meet those inside
$S$ exactly in $E_T$. It therefore suffices to bound $\PP(A_{S'}\mid E_T=e)$ uniformly in
$e$. Condition further on all edges inside $S'$ \emph{except} those between $T$ and
$S'\setminus T$. If $A_{S'}$ holds with common degree $D$, then every $v\in S'\setminus
T$ satisfies $\sum_{u\in T}\xi_{uv}=D-r_v$, where $r_v$ is determined by the conditioned
edges. The sums $\sum_{u\in T}\xi_{uv}$, $v\in S'\setminus T$, involve pairwise disjoint
sets of edges, hence are independent, and each is a sum of $\kappa$ independent
Bernoullis with parameters in $[p_0,1-p_0]$. By Lemma~\ref{lem:anti} and a union over the
at most $y$ possible values of $D$,
$\PP(A_{S'}\mid\cdot)\le y\,(C_1/\sqrt\kappa)^{y-\kappa}$, uniformly in everything
conditioned on. Averaging completes the proof.
\end{proof}

\begin{lemma}[Hypergeometric tail]\label{lem:hyp}
Let $S,S'$ be independent uniform $y$-subsets of an $M$-set with $M\ge3y$, and
$h(\kappa):=\PP(|S\cap S'|=\kappa)$. Then
$h(\kappa)\le\bigl(e\tilde\mu/\kappa\bigr)^\kappa$ with
$\tilde\mu:=y^2/(M-2y)$.
\end{lemma}

\begin{proof}
By the union bound, $h(\kappa)\le\PP(|S\cap S'|\ge\kappa)\le
\binom y\kappa\max_{|A|=\kappa}\PP(A\subseteq S')$, and for a fixed $\kappa$-set $A$,
$\PP(A\subseteq S')=\prod_{i=0}^{\kappa-1}\frac{y-i}{M-i}\le
\bigl(\frac{y}{M-2y}\bigr)^\kappa$; with $\binom y\kappa\le(ey/\kappa)^\kappa$ the claim
follows.
\end{proof}

\subsection{Proof of Theorem~\ref{thm:cond}}

\begin{proof}[Proof of Theorem~\ref{thm:cond}]
Fix $\epsilon\in(0,1/8)$, $w=n^{-1/4}$, $y=\lceil n^{1/2-\epsilon}\rceil$, and work
conditionally on a type profile in the a.a.s.\ event on which $M=n^{3/4}(1+o(1))$ and
Hypotheses~\ref{hyp:llt} and~\ref{hyp:cor} hold. Note
$\tilde\mu=y^2/(M-2y)=n^{1/4-2\epsilon}(1+o(1))$ and set
$\kappa_3:=\lceil\tilde\mu\log n\rceil$ and $\kappa_2:=\lfloor y/\sqrt{\log n}\rfloor$,
so $\kappa_3<\kappa_2$ for large $n$.

By Hypothesis~\ref{hyp:llt}, $\PP(A_S\mid x)>0$ for every $S$, so the \emph{weighted}
count
\[
  \widetilde X:=\sum_{S\subseteq B,\,|S|=y}\frac{\mathbf 1_{A_S}}{\PP(A_S\mid x)}
\]
is well defined; weighting removes any dependence on the variation of
$\PP(A_S\mid x)$ across $S$ (this variation was a gap in the v1 argument). Then
$\E[\widetilde X\mid x]=N:=\binom My$ exactly, and
\[
  \E[\widetilde X^2\mid x]=\sum_{S,S'}\rho(S,S'),\qquad
  \rho(S,S'):=\frac{\PP(A_S\cap A_{S'}\mid x)}{\PP(A_S\mid x)\PP(A_{S'}\mid x)},
\]
so that, grouping by $\kappa=|S\cap S'|$ and writing
$\bar\rho(\kappa)$ for the maximum of $\rho$ at overlap $\kappa$,
$\E[\widetilde X^2\mid x]\le N^2\sum_{\kappa=0}^{y}h(\kappa)\bar\rho(\kappa)$
with $h$ as in Lemma~\ref{lem:hyp}. We bound the sum in three ranges.

\emph{Range $\kappa\le\kappa_3$.} For $\kappa\le1$ the events involve disjoint edge
sets, so $\rho=1$. For $2\le\kappa\le\kappa_3$, Hypothesis~\ref{hyp:cor} together with
$\kappa^2/y\le\kappa_3^2/y=O(n^{-3\epsilon}\log^2n)\to0$ gives
$\bar\rho(\kappa)\le1+o(1)$ uniformly. Hence the whole range contributes
$\sum_{\kappa\le\kappa_3}h(\kappa)\bar\rho(\kappa)\le(1+o(1))\sum_\kappa h(\kappa)
=1+o(1)$.

\emph{Range $\kappa_3<\kappa\le\kappa_2$.} By Hypothesis~\ref{hyp:cor} and
Lemma~\ref{lem:hyp}, using $\log(\kappa/\tilde\mu)\ge\log\log n$ on this range and
$e^{C\kappa^2/y}\le e^{C\kappa/\sqrt{\log n}}$,
$h(\kappa)\bar\rho(\kappa)\le e^{-\kappa(\log\log n-1)}\cdot(1+o(1))\,
e^{C\kappa/\sqrt{\log n}}\le e^{-\kappa}$ for large $n$; the contribution is
$\sum_{\kappa>\kappa_3}e^{-\kappa}=o(1)$.

\emph{Range $\kappa_2<\kappa\le y$.} By Lemma~\ref{lem:bigoverlap} and
Hypothesis~\ref{hyp:llt} (with the drift bound $o(y)$ as in
Proposition~\ref{prop:first}),
\[
  \bar\rho(\kappa)\le
  y\Bigl(\frac{C_1}{\sqrt\kappa}\Bigr)^{y-\kappa}\cdot
  \Bigl(\frac{\sqrt y}{c}\Bigr)^{y-1}e^{o(y)} .
\]
Taking logarithms and using $\log\kappa\ge\log y-\tfrac12\log\log n$ on this range,
\[
  -\frac{y-\kappa}2\log\kappa+\frac y2\log y
  =\frac\kappa2\log y+\frac{y-\kappa}2(\log y-\log\kappa)
  \le\frac\kappa2\log y+\frac y4\log\log n,
\]
while Lemma~\ref{lem:hyp} gives
$\log h(\kappa)\le-\kappa\bigl(\log(\kappa/\tilde\mu)-1\bigr)\le
-\kappa\bigl(\tfrac14+\epsilon\bigr)\log n\,(1-o(1))$, since
$\kappa/\tilde\mu\ge\kappa_2/\tilde\mu=n^{1/4+\epsilon-o(1)}$. Since
$\tfrac\kappa2\log y=\kappa\bigl(\tfrac14-\tfrac\epsilon2\bigr)\log n\,(1+o(1))$,
\[
  \log\bigl[h(\kappa)\bar\rho(\kappa)\bigr]\;\le\;
  -\tfrac32\epsilon\,\kappa\log n\,(1-o(1))+\tfrac y4\log\log n+O(y)
  \;\le\;-\epsilon\,y\sqrt{\log n}
\]
for large $n$, using $\kappa\ge\kappa_2=y/\sqrt{\log n}$, so that
$\kappa\log n\ge y\sqrt{\log n}\gg y\log\log n$; the $O(y)$ term collects
$(y-\kappa)\log C_1$, $(y-1)\log(1/c)$, $\log y$ and the $o(y)$ drift, and is absorbed
the same way, since $y=o(\kappa\log n)$ uniformly on this range. Summing over the at most $y$ values of
$\kappa$ in this range gives $o(1)$.

Altogether $\E[\widetilde X^2\mid x]\le(1+o(1))N^2$, and by Cauchy--Schwarz
(Paley--Zygmund at zero),
\[
  \PP\bigl(F(G)\ge y\mid x\bigr)\ \ge\ \PP(\widetilde X>0\mid x)\ \ge\
  \frac{(\E[\widetilde X\mid x])^2}{\E[\widetilde X^2\mid x]}\ \ge\ 1-o(1).
\]
Since the conditioning event has probability $1-o(1)$ over the types, the theorem
follows.
\end{proof}

\begin{remark}[Asymptotic character]
Theorem~\ref{thm:cond} is purely asymptotic: in the large-overlap range the $o(1)$'s
decay like powers of $1/\sqrt{\log n}$, so the estimates above acquire numerical content
only at very large $n$. We flag this because the paper is otherwise careful to separate
finite-$n$ computation from asymptotic assertion.
\end{remark}

\begin{remark}[Consistency of scales]
The window width $w=n^{-1/4}$ and size $y\approx n^{1/2}$ found by the optimisation of
Remark~\ref{rem:constant} match the scale $s^2=\Theta(1)$, i.e.\
$|y_i|\approx k^{-1/2}$ over $k\approx n^{1/2}$ vertices, that dominates the Gaussian
integral of Section~\ref{sec:mech}. The two sides of the program identify the same
critical configuration; this consistency is a heuristic, not evidence of correctness of
either side.
\end{remark}

\section{Numerical evidence}\label{sec:num}

The mechanism behind Conjecture~\ref{conj:main} predicts, for the model with types
uniform on $[\frac12-h,\frac12+h]$ and $W(x,y)=(x+y)/2$, that optimising the window
against the drift gives
\[
  F(n,h)\;\asymp\;\min\bigl(n^{2/3},\ \sqrt{n/h}\bigr),
\]
the two branches crossing at $h=n^{-1/3}$; the first branch is the $G(n,p)$ behaviour of
\cite{KSW11}, toward which the model degenerates as $h\to0$. (This displayed crossover is
also the correct form of a statement garbled in v1, which asserted
``$F\ge cL^{-1/2}\sqrt n$ recovers $\Theta(n^{2/3})$ as $L\to0$'' --- literally false as
written; $L$ must moreover be read as the \emph{least} Lipschitz constant, since every
graphon is $L$-Lipschitz for every larger $L$.)

We computed $F$ exactly, certified by a constraint solver, at $108$ points of the
complete grid $n\in\{28,32,36,40,44,48\}$, $h\in\{0,0.05,0.12,0.20,0.30,0.40\}$, three
certified instances per cell (the figure legend denotes the order by $m$).
Figure~\ref{fig:trans} and Table~\ref{tab:trans} report the outcome: the ratio
$F/\min(n^{2/3},\sqrt{n/h})$ has mean $1.295$; per-instance values lie in $[1.08,1.65]$
and per-cell means in $[1.19,1.43]$, \emph{including across the change of branch}. At
$n=48$, for instance, the prediction switches from $n^{2/3}=13.2$ at $h=0.20$ to
$\sqrt{n/h}=11.0$ at $h=0.40$, and the measured $F$ falls correspondingly from $16$ to
$14$. Comparing instead with the conjectured constant, $F\approx\sqrt{2en}$, the measured
values at $h\ge0.25$ give ratios between $0.75$ and $1.10$ for $32\le n\le48$.

Three cautions, which we owe to the critique of v1. First, computations at orders
$n\le48$ can illustrate a predicted crossover but cannot validate
Hypotheses~\ref{hyp:llt}--\ref{hyp:cor}, the conjectured constant, or any asymptotic
statement. Second, the superimposed curve in Figure~\ref{fig:trans} is scaled by a global
factor $\approx1.3$ which is \emph{fitted}, not predicted; what is parameter-free is the
shape, in particular the location of the crossover and the flatness of the ratio across
it. Third, selection effects: averages over certified instances could in principle be
biased by which instances time out; here the issue does not arise because the grid is
complete --- all $108$ instances certified, none timed out --- with every witness
verified by hand and an \textsc{infeasible} verdict counting as proof. (The v1 text cited
a stale count of $106$ points on a grid extending to order $56$; the numbers, table and
figure here are regenerated directly from the certified dataset.)

\begin{figure}[ht]
\centering
\includegraphics[width=0.92\textwidth]{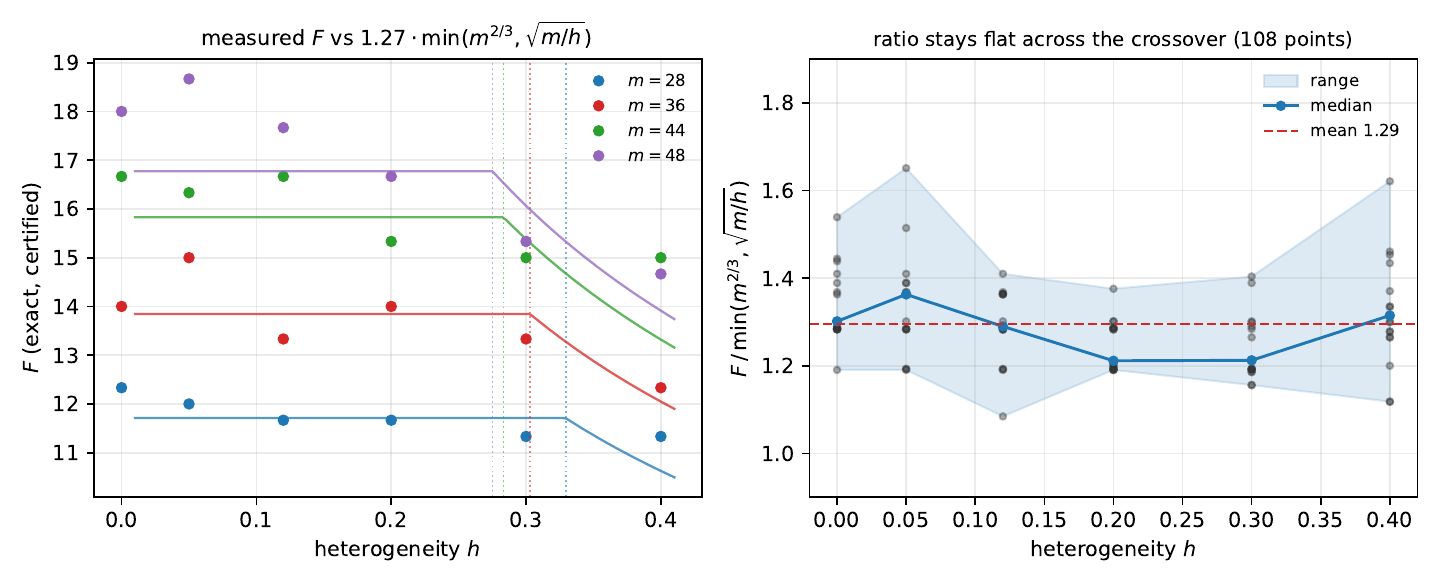}
\caption{Left: exact $F$ against the heterogeneity $h$ for several orders, with the
prediction $\min(n^{2/3},\sqrt{n/h})$ superimposed, scaled by a fitted global factor
$\approx1.3$ (the fit shown is $1.27$); the kink is the crossover between the
entropy-limited and the drift-limited branch. Right: the ratio $F/\min(n^{2/3},\sqrt{n/h})$ for all $108$ certified
measurements, with the median per value of $h$ and the observed range; it stays flat
across the crossover, which is what a correct parameter-free \emph{shape} should do.}
\label{fig:trans}
\end{figure}

\begin{table}[ht]
\centering
\begin{tabular}{rcccccc}
\hline
$n$ & $h=0$ & $h=0.05$ & $h=0.12$ & $h=0.20$ & $h=0.30$ & $h=0.40$\\ \hline
$28$ & $1.34$ & $1.30$ & $1.26$ & $1.26$ & $1.23$ & $1.35$\\
$32$ & $1.29$ & $1.32$ & $1.26$ & $1.19$ & $1.29$ & $1.23$\\
$36$ & $1.28$ & $1.38$ & $1.22$ & $1.28$ & $1.22$ & $1.30$\\
$40$ & $1.40$ & $1.34$ & $1.31$ & $1.23$ & $1.24$ & $1.27$\\
$44$ & $1.34$ & $1.31$ & $1.34$ & $1.23$ & $1.24$ & $1.43$\\
$48$ & $1.36$ & $1.41$ & $1.34$ & $1.26$ & $1.21$ & $1.34$\\
\hline
\end{tabular}
\caption{$F/\min(n^{2/3},\sqrt{n/h})$, mean over the three certified instances per cell.
Entries are stable across a change of the binding branch of the prediction.}
\label{tab:trans}
\end{table}

\section{Concluding remarks}\label{sec:concl}

\subsection{Why $\sqrt n$ is the ceiling of this scheme}
Heuristically: the union bound over $k$-sets contributes $e^{k\log(n/k)}$ while the
per-set failure probability is $e^{-\Theta(k\log k)}$ (one $\sqrt{\ }$-fluctuation per
vertex); these balance at $n\asymp k^2$ irrespective of the curvature constants. Any
first-moment argument over this family therefore stops at $f(n)\le A\sqrt n$, and the
content of Theorem~\ref{thm:DM} together with Conjecture~\ref{conj:main} would be that
the optimal $A$ is $\sqrt{2e}$.

\subsection{Consequences for the Erd\H{o}s--Fajtlowicz--Staton problem}
If Conjecture~\ref{conj:main} holds, then within the family of
Definition~\ref{def:model} the answer is $\Theta(\sqrt n)$ with optimal constant
$\sqrt{2e}$, and any improvement of the upper bound for $f(n)$ must come from a
construction outside that family --- bearing in mind the scope caveat of
Section~\ref{sec:intro}. The proof strategy for the conjecture uses two ingredients: a
block with controlled degree drift, which is available in any graph by pigeonholing
degree bands, and a local limit theorem for the joint degree vector
(Hypothesis~\ref{hyp:llt}), which uses independence of the edges. For a deterministic $G$
one may still randomise the \emph{subset}; what is missing is the joint anticoncentration
statement, with the additional twist that the condition is imposed only on the
coordinates lying in $S$.

A small experiment locates the difficulty. Counting exhaustively, for graphs on $20$
vertices, the number $N_y$ of $y$-subsets inducing a regular subgraph and comparing it
with the generic first-moment prediction $\binom {20}y(2\pi\sigma^2)^{-(y-1)/2}$ (with
$\sigma^2$ the variance of a degree, here $\approx y/4$), one finds: for $G(20,1/2)$ the prediction is accurate to within a factor $e^{\pm0.7}$ for
every $y\le8$; forcing all degrees to be even changes nothing; but for a disjoint union
of cliques of sizes $2,\dots,6$ there are \emph{exactly zero} regular $7$-subsets --- a
regular induced subgraph there is a disjoint union of equal cliques, and $7=j(d+1)$
forces $d+1\in\{1,7\}$, neither of which is available --- while $y=8=4\cdot2=2\cdot4$
admits $4595$. In these experiments the generic count is correct exactly when no arithmetic obstruction
is present.

It is worth recording that the relevant technology exists in a neighbouring problem.
Kwan, Sah, Sauermann and Sawhney \cite{KSSS23} control the distribution of $e(G[U])$ for
a random subset $U$ of a Ramsey graph --- and hence resolve the Erd\H{o}s--McKay
conjecture --- by means of an additive-structure dichotomy on the degree sequence,
formulated through the regularized least common denominator: if it is small, the vertices
fall into few buckets of almost equal degree, and if it is large, a local central limit
theorem is available. The first alternative is exactly the regime in which a graph is
close to a blow-up, where large regular induced subgraphs are easy to exhibit; the second
is exactly what Hypothesis~\ref{hyp:llt} requires. The gap to be bridged is that
\cite{KSSS23} controls a single quadratic statistic, whereas a regular induced subgraph
imposes $|S|$ simultaneous conditions on the degree vector.

\subsection*{Code availability}
The numerical checks reported here --- the determinant consistency check of
Section~\ref{sec:mech} and the certified computation of the $108$ values of $F$ behind
Figure~\ref{fig:trans} and Table~\ref{tab:trans} --- together with the raw data, are
available at
\url{https://gist.github.com/arielelevy/c9cf5b142951db0c7d179a3656098835}.

\subsection*{Acknowledgements}
This version owes its existence to a detailed critique of the first version by
P.~W.~Dyson and B.~D.~McKay, whom I thank for their care and generosity --- including
their decision to leave the lower-bound direction to this program. The positivity
argument of Lemma~\ref{lem:pos} and the counterexamples of Remarks~\ref{rem:c5}
and~\ref{rem:specfalse} are theirs. The analysis and the numerical verifications were
developed with substantial assistance from Claude (Anthropic); all derivations were
checked by hand and against exact computations.

\end{document}